\documentclass[10pt]{article}
\RequirePackage[T1]{fontenc}
\RequirePackage{amsfonts,amsthm,amsmath}
\theoremstyle{plain}

\usepackage{a4wide}
\newtheorem{theo}{Theorem}
\newtheorem{Hyp}{Hypothesis}

\newtheorem{lemme}[theo]{Lemma}
\newtheorem{prop}[theo]{Proposition}
\newtheorem*{remi}{Remark}
\def\a{{\alpha}}
\def\1{{\mathbf{1}}}

\def\p{{\mathbf P}}
\def\P{{\mathbf P}}
\def\e{{\mathbf E}}

\def\eps{{\varepsilon}}

\begin{document}
\date{\hspace*{0mm}}
\title{\textbf{Deviation probability bounds for fractional martingales and related remarks}}
\author{Bruno \textsc{Saussereau}
\footnote{Laboratoire de Math\'{e}matiques de Besan\c{c}on, CNRS, UMR 6623;
16 Route de Gray,
25030 Besan\c{c}on cedex, France.
\texttt{bruno.saussereau@univ-fcomte.fr}}
}
\maketitle

\begin{center}
\begin{minipage}{0.9\textwidth}
\begin{center}\textbf{Abstract}\end{center}
In this paper we prove exponential inequalities (also called Bernstein's inequality) for fractional martingales. As an immediate corollary, we will discuss weak law of large numbers for fractional martingales under divergence assumption on the $\beta-$variation of the fractional martingale. A non trivial example of application of this convergence result is proposed.
 \\
%
\hrule
\vspace{0.1cm}
\footnotesize{
\textbf{Keywords: }Fractional martingales, exponential inequality, law of large numbers \\ %
\textbf{AMS 2000 Subject classifications: 60G22, 60G48, 60B12}
}
\end{minipage}
\end{center}
%
\section{Introduction}
The notion of fractional martingales has been introduced in \cite{hns} where the author proved an extension of L\'evy's characterization theorem to the fractional Brownian motion. The purpose of this short communication is to investigate exponential inequalities of Bernstein's type and their applications to laws of large numbers for fractional martingales.

More precisely when we fix $\alpha\in (-\frac{1}2{,}\frac{1}2)$, if $M=(M_t)_{t\ge 0}$ is a continuous local martingale, the process $M^{(\a)}=(M_t^{(\alpha)})_{t\ge 0}$ defined by
$ M_t^{(\alpha)} = \int_0^t (t-s)^\a dM_s$ will be called a fractional martingale (provided that the above integral exists).
For a fixed time $t$, we can consider the true martingale  $(Z^t_u)_{0\le u\le t}$ defined by $Z^t_u = \int_{0}^u (t-s)^\alpha \xi_s dW_s$. Here $t$ is consider as a fixed parameter for the martingale $Z^t$.
As a consequence of the classical exponential inequality, one can easily obtain some deviation probability bounds. For example if $\alpha<0$, it is clear that
\begin{align*}
    \mathbf P \left ( | M^{(\alpha)}_t |\ge u\ ,\ \mbox{$\int_0^t |\xi_\tau|^{2}d\tau \le \nu_t $ }\right ) & \le 2 \exp \left ( -  \frac{u^2}{4 \ t^{2\a} \ \nu_t}\right ).
\end{align*}
Hence it is easy to prove some exponential inequalities when $t$ is fixed (see also Remark \ref{rem+}).

Our aim is to investigate some deviation bounds for $\sup_{0\le s\le t} |M_s^{(\alpha)}|$. This will be no more a straightforward application of the result in the martingale case when $\alpha=0$. We recall that for the martingale $M^{(0)}$, if it vanishes at time $t=0$, then
\begin{align}\label{revuz}
 \p \left ( \sup_{s\le t} |M_s| \ge at  \right ) \le \ 2\exp \left ( -\frac{a^2t}{2c}\right ) ,
\end{align}
if $c$ is a constant such that $\langle M \rangle_t \le ct $ for all $t$ (see \cite{ry} Exercice 3.16, Chapter 4).

And so our work will concern the extension of exponential inequalities similar to \eqref{revuz} for fractional martingales.
Since on any time interval, the process $M^{(\a)}$ has finite nonzero variation of order $\beta=2/(1+2\a)$
we shall try to use some quantities related to the $\beta-$variation in the statement of our Bernstein's type inequality. This will represent the main result of this paper and it is the purpose of Theorem \ref{bernstein}.

As an application of these exponential inequalities, we will have a discussion around weak law of large numbers for fractional martingales. We think that using our result on the deviation probability bounds for fractional martingale is a first step to the study of law of large numbers for fractional martingales. So in Proposition \ref{wlln} we establish  that $\sup_{0\le s\le t} |M^{(\a)}_s|/\langle M^{(\a)} \rangle_{\beta,t} $ tends to $0$ in probability provided that the $\beta-$variation $\langle M^{(\a)} \rangle_{\beta,t}$ tends to infinity faster than $t^a$ for one $a>0$. Of course this is not a classical condition but we present a non trivial application of this weak law of large numbers in Proposition \ref{apply}. To end this discussion, a related convergence result is given in Proposition \ref{conv00} under the more conventional assumption on the divergence of the quadratic variation of the underlying martingale. More precisely, we shall prove that $\int_0^t (t-s)^\alpha\xi_sdW_s/\int_0^t (t-s)^{2\alpha}\xi_s^2ds$ tends to 0 almost surely when $\alpha>0$ and under the assumption that $\int_0^{+\infty} \xi_s^2 ds =+\infty$ almost-surely. In a sense, the law of large numbers for the martingale $(\int_0^t\xi_sdW_s)_{t\ge 0}$ is transferred to its fractional martingale $M^{(\alpha)}$.

The paper is organized as follows. In the following section we precise our notations and we state our results. Two proofs will be given in Section \ref{proofs} and Section \ref{proof_conv00}.
\section{Notations and main results}\label{nota}
We follow the terminology of \cite{hns}. Let $(\Omega{,}\mathcal F{,}\p)$ be a complete probability space equipped with a continuous filtration $(\mathcal F_t)_{t\ge 0}$ such that $\mathcal F_0$ contains the $\P-$negli\-geable events.
Let $\beta\ge 1$ and let $X=(X_t)_{t\ge 1}$ be a continuous adapted process. The $\beta-$variation of $X$ on the time interval $[0{,}t]$ is denoted by $\langle X\rangle_{\beta,t}$ and is defined as the limit in probability (if it does exist) of
$$ S^{[0{,}t]}_{\beta,n} := \sum_{i=1}^n | X_{t_i^n}-X_{t_{i-1}^n}|^\beta$$
where for $i=0,...,n$, $t_i^n =\frac{i}n \times t$. If the convergence holds in $\mathbf{L}^1$, we say that the $\beta-$variation exists in $\mathbf{L}^1$. A parameter $\a \in (-\frac{1}2 {,} \frac{1}2 ) $ is fixed and we denote
\begin{align*}
&\beta = \frac{2}{1+2\a} .
\end{align*}
We notice that $\beta\in (1{,}+\infty)$ and $\beta>2$ when $\a<0$. Let $M^{(\a)}=(M^{(\a)}_t)_{t\ge 0}$ a fractional martingale of order $\a$. This means that $M^{(\a)}$ is a continuous $\mathcal F_t-$ada\-pted process such that there exists a continuous local martingale $M=(M_t)_{t\ge 0}$ with $ \int_0^t (t-s)^{2\a}d \langle M\rangle_s <\infty  $ a.s. for all $t\ge 0$,  and
\begin{align}\label{m-alpha}
 M^{(\a)}_t & = \int_0^t (t-s)^\a dM_s\ .
\end{align}
If $\a\in (0{,}\frac{1}2)$, the above integral always exists as a Riemann-Stieltjes integral. In order to ensure the existence $M^{(\a)}$ when $\a\in (-\frac{1}2{,}0)$, we assume the following  hypothesis in all the sequel.
\begin{Hyp}\label{hyp-M}
The continuous local martingale $M$ is of the form
$$ M_t = \int_0^t \xi_s \,dW_s$$ where $W=(W_t)_{t\ge 0}$ is a $\mathcal F_t-$Brownian motion and $\xi=(\xi_t)_{t\ge 0}$ is a progressively measurable process such that for all $t\ge 0$
$$
\left\{
  \begin{array}{ll}
    \displaystyle{ \int_0^t \e\big (|\xi_s|^{\beta'} )}ds <\infty\quad \hbox{for some $\beta'>\beta$,} & \hbox{if $\a<0$;} \\
            & \\
    \displaystyle{  \int_0^t  \e\big (|\xi_s^2 |\big )ds} <\infty\ ,  & \hbox{if $\a>0$.} \\
  \end{array}
\right.
$$
\end{Hyp}
Under Hypothesis \ref{hyp-M}, the integral appearing in \eqref{m-alpha} always exists as a Riemann-Stieltjes integral. This is a consequence of \cite[Lemma 2.2]{hns} and the fact that the trajectories
of $M$ are $\alpha'-$H\"{o}lder continuous on finite interval. Moreover, by Theorem 2.6 and Remark 2.7 of \cite{hns}, the $\beta-$variation of $M^{(\a)}$ exists in $\mathbf{L}^1$ and
$$ \langle M^{(\a)} \rangle_{\beta,t} = c_\a \ \int_0^t |\xi_s|^\beta ds $$
where $c_\a$ depends only on $\a$. The explicit form of the constant $c_\alpha$ is given in \cite{hns} but this is not important in our work.

Nevertheless we stress the point that under Hypothesis \ref{hyp-M}, the expression of $M^{(\a)}$ is given by
\begin{align}\label{m-alpha-bis}
 M^{(\a)}_t & = \int_0^t (t-s)^\a \,\xi_s \, dW_s\ .
\end{align}
Moreover, using Hölder's inequality  one deduces the following relations between the $\beta-$va\-ria\-tion of $M^{(\a)}$ and the quadratic variation of the underlying martingale $M$:
\begin{align*}
 &    \left\{
      \begin{array}{ll}
        \langle M\rangle_t  & \le c_\a^{-2/\beta}\ t^{\frac{\beta-2}{\beta}}\  \langle M^{(\a)} \rangle_{\beta,t}^{2/\beta}\hfill{\qquad\hbox{when $\a<0$;}} \\
         &  \\
         \langle M^{(\a)} \rangle_{\beta,t}& \le c_\a\ t^{\frac{2-\beta}2} \ \langle M\rangle_t^{\beta/2} \hfill{\hbox{when $\a>0$.}}
      \end{array}
    \right.
\end{align*}
Our main result which is a generalization of Bernstein's inequality to fractional martingales is stated in the next theorem.
\begin{theo}\label{bernstein}
We assume Hypothesis \ref{hyp-M}. We denote ${C_t}=2 + 2^{1/2} t^2 $. For any positive function $t\mapsto \nu_t$ and any $L\ge 1$ the following exponential inequalities hold.
\begin{itemize}
\item[(i)] When $\alpha<0$  we have
\begin{align}\label{inegexp}
&    \mathbf P \left ( \sup_{0\le s\le t}| M^{(\alpha)}_s |\ge L\,{c_1}\,t^{\frac{\beta'-\beta}{2\beta\beta'}} \,\nu_t^{1/2}  \ ,\ \mbox{$\big (\int_0^t |\xi_\tau|^{\beta'}d\tau \big )^{2/\beta'}\le \nu_t $ }\right )\le C_t \exp \left \{ - \frac{{\kappa}^2\,L^2}{t^{\frac{\beta'-\beta}{\beta\beta'}}}   \right \}\
\end{align}
with $c_1$ defined in \eqref{c1} and $\kappa^2 = 4\pi (\beta\beta'/(\beta'-\beta))^3$.
\item[(ii)] When $\a>0$, for any $\eps\in (0{,}\alpha)$ it holds  that
\begin{align}\label{inegexp2}
    \mathbf P \left ( \sup_{0\le s\le t}| M^{(\alpha)}_s |\ge \,L\, 2^6\,\kappa\,t^{\alpha-\eps} \,\nu_t^{1/2} \ ,\ \mbox{$\int_0^t |\xi_\tau|^{2}d\tau \le \nu_t $ }\right ) & \le C_t \exp \left \{ - \frac{\kappa^2\,L^2}{t^{2(\alpha-\eps)}}   \right \}
\end{align}
with $\kappa = (\pi/2)^{1/2} \eps^{-3/2}$.
\item[(iii)]
If we assume that the process $\xi$ is bounded by $c_\infty$ almost-surely, then for any $\alpha\in (-\frac{1}2{,}\frac{1}2)$ and any $\eps\in (0{,}\frac{1}2 +\alpha)$
\begin{align}\label{inegexp3}
    \mathbf P \left (\sup_{0\le s\le t} | M^{(\alpha)}_s |\ge L\, 2^6\,\kappa\,c_\infty\,t^{1/2+\alpha-\eps} \right ) & \le  C_t \exp \left \{ - \frac{\kappa^2\,L^2}{t^{1+2\alpha-2\eps}}   \right \}\ ,
\end{align}
with $\kappa = (\pi/2)^{1/2} \eps^{-3/2}$.
\end{itemize}
\end{theo}

Formally, the above inequalities are consistent (asymptotically when $t$ grows to infinity) with the classic ones recalled in \eqref{revuz} when $\alpha=0$ (or equivalently $\beta = \beta'=2$). For example, one can put $L\asymp t^{1/2+\eps}$ with $\eps=1/4$ in \eqref{inegexp2}.
\begin{remi}\label{rem+}
The above result have a straightforward proof if we are interested by exponential inequalities without the supremum with respect
to $s\in [0{,}t]$. For example to show that
\begin{align*}
    \mathbf P \left ( | M^{(\alpha)}_t |\ge u\ ,\ \mbox{$\big (\int_0^t |\xi_\tau|^{\beta'}d\tau \big )^{2/\beta'}\le \nu_t $ }\right ) & \le 2 \exp \left ( -  \frac{u^2}{4 \ C_{\beta,\beta'}\ t^{2(\beta'-\beta)/\beta\beta'} \ \nu_t}\right )
\end{align*}
when $\alpha<0$, it suffices to remark that, by Hölder's inequality, $(\int_{0}^t |\xi_\tau|^{\beta'}d\tau )^{2/\beta'}\le \nu_t$ implies that $$\int_0^t (t-s)^{2\alpha} \xi_s^2 ds \le C_{\beta,\beta'} t^{2(\beta-\beta')/\beta\beta'} \nu_t .$$ Thus the inequality is a consequence of the classical exponential inequality when one considers the martingale $(Z^t_u)_{0\le u\le t}$ defined by $Z^t_u = \int_{0}^u (t-s)^\alpha \xi_s dW_s$ ($t$ is consider as a fixed parameter).
\end{remi}
The uniform deviations stated in theorem \ref{bernstein} are a little bit more complicated than the one we described in the above remark. Their proofs are postponed in Section \ref{proofs}.

As a corollary of the above theorem, we obtain a weak law of large numbers for fractional martingales.
\begin{prop}\label{wlln}
Under Hypothesis \ref{hyp-M}, let $M^{(\alpha)}$ be a fractional martingale with $\alpha\in (-\frac{1}2{,}\frac{1}2)$ having the expression $M^{(\alpha)}_t = \int_0^t (t-s)^\a \xi_s dW_s$. We assume that the process $\xi$ is bounded (by a constant $c_\infty$). Then we have the following weak law of large numbers: suppose that there exists $a>0$ such that
$$\lim_{t\to\infty}\frac{\langle M^{(\a)} \rangle_{\beta,t}}{t^a}=+\infty\ \text{almost-surely}$$
then the following convergence holds in probability
\begin{align*}
\frac{\sup_{0\le s\le t} |M^{(\a)}_s|}{\langle M^{(\a)} \rangle_{\beta,t}} \xrightarrow[t\to\infty]{}  0\ .
\end{align*}
\end{prop}
\begin{proof}
With $\eta>0$ we use \eqref{inegexp3} to write that
\begin{align}\label{lo}
    \P \left ( \frac{\sup_{0\le s\le t} |M^{(\a)}_s|}{\langle M^{(\a)} \rangle_{\beta,t}} \ge \eta \right )
& \le \P \left (  \frac{\sup_{0\le s\le t} |M^{(\a)}_s|}{\langle M^{(\a)} \rangle_{\beta,t}}\ge \eta \ ,\  \langle M^{(\a)} \rangle_{\beta,t} \ge t^a \right ) \nonumber \\ %
& \hspace{15mm}+ \P \left (\langle M^{(\a)} \rangle_{\beta,t} \le t^a \right ) \nonumber \\
& \le \P \left ( \sup_{0\le s\le t} |M^{(\a)}_s| \ge \eta \ t^a \right ) + \P \left ( \langle M^{(\a)} \rangle_{\beta,t} \le t^a \right )   \nonumber \\
&\le C_t\exp\left ( -\frac{\eta^2\,t^{2a}}{128 \, c_\infty^2\, t^{2(1+2\a-2\eps)}}  \right ) +\P \left ( \langle M^{(\a)} \rangle_{\beta,t} \le t^a \right ) .
\end{align}
It suffices to choose $\eps$ closed to $1/2+\alpha$ such that $a>1+2\alpha-2\eps$ and the first term in the right hand side of \eqref{lo} tends to $0$ as $t$ goes to infinity. The second term in \eqref{lo} tends to $0$ because
$\lim_{t\to\infty}t^{-a}\langle M^{(\a)} \rangle_{\beta,t}=+\infty$ almost-surely.
\end{proof}
The following proposition provides a non trivial example of application of the above result.
\begin{prop}\label{apply}
For $H\in (0{,}1)$, let $B^H=(B^H_t)_{t\ge 0}$ be a fractional Brownian motion (see \cite{n} for details) adapted with respect to the filtration $(\mathcal F_t)_{t\ge 0}$ and let $\Phi$ be a bounded continuous function from $\mathbf{R}$ to $\mathbf{R}$. With $\alpha\in (-\frac{1}2{,}\frac{1}2)$, the fractional martingale $N^{(\a)}$ defined by $N^{(\alpha)}_t = \int_0^t (t-s)^\a \Phi(B_s^H)dW_s$ satisfies the weak law of large numbers
\begin{align}\label{wapp}
    \frac{\sup_{0\le s\le t}|N_s^{(\a)}|}{\langle N^{(\a)} \rangle_{\beta,t}} \xrightarrow[t\to\infty]{\P} 0 \ .
\end{align}
\end{prop}
\begin{proof}
As regard to \eqref{lo}, one have to find $a>0$  such that
\begin{align*}
&\P \left ( \langle N^{(\a)} \rangle_{\beta,t} \le t^a \right ) \xrightarrow[t\to\infty]{}  0\ .
\end{align*}
For that sake, we will make use of the local time $L^H(t,y)$ of $B^H$ at $y\in \mathbf{R}$ defined heuristically for $t\ge 0$ as $$ L^H(t,y) = \int_0^t \delta_y(B^H_s)ds. $$
It is known (see \cite{berman2,gh}) that $(t,y)\mapsto L^H(t,y)$ exists and is jointly continuous in $(t,y)$. By
the self-similarity property of the fractional Brownian motion, the distributions of $L^H(t,y)$ and $t^{1-H}L^H(1,yt^{-H})$ are equal. Using the occupation times formula, we may write that
\begin{align*}
    \int_0^t |\Phi(B^H_s)|^\beta ds & = t\int_0^1 |\Phi(B^H_{tu})|^\beta du \\ %
& \stackrel{\mathrm{d}}{=} t\int_0^1 |\Phi(t^{H}B^H_u)|^\beta du   \qquad \mbox{(in distribution)} \\
& = t\int_\mathbf{R} |\Phi(t^{H}y)|^\beta L^H(1,y) dy \\
& = t^{1-H}\int_\mathbf{R} |\Phi(z)|^\beta L^H(1,zt^{-H}) dz\ .
\end{align*}
By the bi-continuity of the local time, we finally obtain
\begin{align*}
     \frac{1}{t^{1-H}}\int_0^t |\Phi(B^H_s)|^\beta ds  \xrightarrow[t\to\infty]{} \left (\mbox{$\int_{\mathbf{R}}|\Phi(z)|^\beta dz$}\right )\ L^H(1,0)
\end{align*}
in distribution. Consequently we have
\begin{align*}
 \P \left ( \langle N^{(\a)} \rangle_{\beta,t} \le t^a \right ) &  = \P \left ( \frac{\langle N^{(\a)} \rangle_{\beta,t} }{t^{1-H}} \le \frac{t^a}{t^{1-H}} \right ) \\
& \xrightarrow[t\to\infty]{}  \p \Big  ( \left (\mbox{$\int_{\mathbf{R}}|\Phi(z)|^\beta dz$}\right )\ L^H(1,0)  \le 0 \Big ) =0
\end{align*}
as soon as $a<1-H$. It remains to remark that such a choice of $a$ is always possible and the convergence \eqref{wapp} is thus a consequence of Proposition \ref{wlln}.
\end{proof}
To end this discussion about the law of large numbers for fractional martingales, one has to mention the following result. It has been used in \cite{s-stat} to investigate asymptotic properties of a nonparametric estimation of the drift coefficient in fractional diffusion.
\begin{prop}\label{conv00}
Let $\xi=(\xi_s)_{s\ge 0}$ is one dimensional, adapted process with respect to the filtration generated by a standard Brownian motion $W=(W_t)_{t\ge 0}$, such that for any $T>0$, $\int_{0}^T \xi_s^2ds <\infty$.

When $\alpha>0$ and $\int_0^{+\infty} \xi_s^2 ds =+\infty$ almost-surely, we have
\begin{align}\label{conv0}
 &\lim_{t\to\infty}\frac{\int_0^t (t-s)^\alpha \,\xi_s \,dW_s}{\int_0^t (t-s)^{2\alpha}\, \xi_s^2 \,ds} =0 \quad\text{almost-surely.}
\end{align}
\end{prop}
we notice that the assumption on the divergence of the quadratic variation of the martingale $(\int_0^t\xi_sdW_s)_{t\ge 0}$ is more common. Nevertheless this result is not a straightforward application of the techniques used in the martingale case when $\alpha=0$. The proof of  \eqref{conv0} is based on a fractional version of the Toeplitz lemma and is postponed in Section \ref{proof_conv00}.
\section{Proof of Theorem \ref{bernstein}}\label{proofs}
Exponential inequalities for continuous martingales have attracted a lot of attention (see for example \cite{cfn,ls,pena}).
Due to our fractional framework, the technics used in the aforementioned works are useless. Our methodology is closed to the one used in \cite[Theorem 2]{nrovira} (see also \cite{rs,sowers}).
That being said we need the following lemma.
\begin{lemme}\label{mayo}
Let $\eps>0$ satisfying $\a <\eps < 1$. Then there exists a constant $C=C_{\a,\eps}$ such that
\begin{align}\label{ineq}
    \left | (u+h)^{\a}-u^{\a} \right | & \le C\ h^{\eps}\ u^{\a-\eps}\ , \ \ \forall u>0,\ h>0\ .
\end{align}
\end{lemme}
\begin{proof}
With $h=xu$, Inequality \eqref{ineq} is equivalent to
$$    \left | 1-(1+x)^{\a} \right |  \le C\ x^{\eps}\ , \ \ \forall x>0\ .$$
According to the cases we need to prove that
$$\left\{
    \begin{array}{ll}
   -1+(1+x)^{\a} & \le C\ x^{\eps},\ \forall x>0\ \text{when $\a>0$} \\ %
     1-(1+x)^{\a} & \le C\ x^{\eps},  \ \forall x>0\ \text{when $\a<0$} \ .
  \end{array}
  \right.
$$
We denote $F_C$ and $G_C$ the functions defined for $x\ge 0$ by
\begin{align*}
F_C(x) & = -1 + (1+x)^{\a}-Cx^\eps\quad \text{and} \\ %
 G_C(x)& = 1 - (1+x)^{\a}-Cx^\eps .
\end{align*}
Then $F_C(0)=G_C(0)= 0 $,
\begin{align*}
F_C'(x) & = \a(1+x)^{\a-1}-C\eps x^{\eps-1}\quad\text{and} \\
G_C'(x) & = -\a(1+x)^{\a-1}-C\eps x^{\eps-1}\ .
\end{align*}
We have to prove that there exists a constant $C$ depending on $\a$ and $\eps$ such that
$$\left\{
    \begin{array}{ll}
      F_C'(x)<0, & \hbox{$\forall x>0$ when $\a>0$;} \\
      G_C'(x)<0, & \hbox{$\forall x>0$ when $\a<0$.}
    \end{array}
  \right.
$$
Further calculations show that if we choose $C$ such that
$$ C \ge \frac{|\a|}{\eps} \ \sup_{x\ge 0}\left \{ \frac{x^{1-\eps}}{(1+x)^{1-\a}}\right \},$$
then \eqref{ineq} is true. It is easy to check that we may find a constant $C$ satisfying the above inequality and that is independent of $\eps$.
\end{proof}
Now we prove Theorem \ref{bernstein}.
\begin{proof}
We follow the arguments developed in \cite{nrovira}. The inequality \eqref{inegexp} will be a consequence of
Chebyshev's exponential inequality involving the random variable $\sup_{0\le s\le t}|M_s^{(\alpha)}|$. So the first step is to apply the Garsia-Rodemich-Rumsey inequality in order to have bounds on this random variable.
With $\Psi(x) = \exp( -x^2/4)$ and $p$ a continuous, non-negative function on $(0,t)$ such that $p(0)=0$, Lemma 1.1 in \cite{grr} reads as follows: for all $0\le r\le s\le t$ we have
\begin{align}\label{garsia}
|M_s^{(\alpha)}-M_r^{(\alpha)}| & \le 8 \ \int_{0}^{|s-r|} \Psi^{-1} \left ( \frac{4B}{y^2} \right ) dp(y)
\end{align}
provided that
\begin{align*}
B & := \int_0^t\int_0^t \Psi\left ( \frac{M_s^{(\alpha)}-M_r^{(\alpha)}}{p(|s-r|)} \right )dsdr <\infty \ .
\end{align*}
The function $p$ will be chosen later. We notice that the function $\Psi^{-1} $ is defined for $u\ge \Psi(0)$ as $\Psi^{-1}(u)  = \sup\{ v; \Psi(v)\le u\}$.
With $\ln^+$ the function defined by $\ln^+(z)=\max (\ln(z),0)$ for $z\ge 0$, the inequality
$B/y^2 \le \exp(\ln^+(B/y^2))$ implies that
$$\Psi^{-1}\left (\frac{B}{y^2}\right )\le 2\left ( \ln^+ \left (\frac{B}{y^2}\right ) \right )^{1/2}.$$
Further calculations show that $$ \left ( \ln^+ \left (\frac{B}{y^2}\right ) \right)^{1/2} \le
2^{1/2} \left \{ \left ( \ln^+(B)\right )^{1/2} +  \left (\ln^+(y^{-2}\right )^{1/2} \right \}\ .$$
Since $M_0^{(\alpha)}=0$, we deduce from \eqref{garsia} that
\begin{align}\label{garsia2}
\sup_{0\le s\le t}| M_s^{(\alpha)}| & \le 2^{9/2} \ \int_{0}^{t} \left \{ \left ( \ln^+(B)\right )^{1/2} +  \left (\ln^+(y^{-2}\right )^{1/2} \right \} dp(y) .
\end{align}
In the following we will need an estimate of the expectation of the random variable $B$. This will be possible
because  a martingale with bounded quadratic variations will appear by means of the increments of $M^{(\a)}$. We fix $t$ and for any $0\le r <s <t $ we write
$$ M^{(\alpha)}_s - M^{(\alpha)}_r = \int_0^s g_{s,r}(\tau) \,dW_\tau$$ with $g_{s,r}(\tau) = \xi_\tau(s-\tau)^{\a} \1_{\{r<\tau\le s\}} + \xi_\tau ( (s-\tau)^{\a}-(r-\tau)^{\a}) \1_{\{\tau\le r\}}$.  We first notice that
\begin{align}\label{inc}
    \int_0^t | g_{s,r}(\tau)|^2 d\tau  =&  \int_r^s (s-\tau)^{2\a} |\xi_\tau|^2 d\tau
  + \int_0^r  ( (s-\tau)^{\a}-(r-\tau)^{\a})^2 |\xi_\tau|^2 d\tau
\end{align}
In order to have some estimates of the quantity $\int_0^t | g_{s,r}(\tau)|^2 d\tau$, we treat different cases according to the sign of $\a$ and according to the assumption we made on the process $\xi$.
\subsubsection*{Case (i):  we assume Hypothesis \ref{hyp-M} and $\a<0$}
With $\beta'>\beta$ from Hypothesis \ref{hyp-M}, we denote $p=\beta'/2>1$ and $q=\beta'/(\beta'-2)$ its conjugate. Let $\eps>0$ to be fixed later. Starting from \eqref{inc}, we use Lemma \ref{mayo} to write
\begin{align}\label{i0}
    \int_0^t | g_{s,r}(\tau)|^2 d\tau
 \le & (s-r)^{2\eps} \left [ \int_r^s (s-\tau)^{2\a-2\eps} \xi_\tau^2 d\tau + \int_0^r (r-\tau)^{2\a-2\eps}\xi_\tau^2 d\tau \right ]  \\ %
\le & (s-r)^{2\eps} \Big [ I_1+I_2\Big ] \label{i1}
\end{align}
with obvious notations.
Now we choose $\eps$ such that $1+2\a q-2\eps q>0$. We remark that such a choice is always possible. Indeed
\begin{align*}
1+2\a q-2\eps q & = \frac{2}{\beta(\beta'-2)}\Big [ (\beta'-\beta)-\eps\beta\beta'\Big ]
\end{align*}
and then choosing $\eps$ of the form $\eps = a(\beta'-\beta)/\beta\beta'$ with $a\in (0{,}1)$, we remark that
$$ 0< 1+2\a q-2\eps q < \frac{2(\beta'-\beta)}{\beta(\beta'-2)}\ .$$
We chose $a=1/2$, henceforth $\eps$ is fixed as
\begin{align*}
\eps &  = \frac{1}2 \frac{\beta'-\beta}{\beta\beta'}\ .
\end{align*}
By H\"{o}lder's inequality we obtain
\begin{align}\label{i11}
    I_1 & \le \left ( \int_r^s (s-\tau)^{(2\a-2\eps)q} d\tau \right )^{1/q} \left ( \int_r^s  \xi_\tau^{2p} d\tau \right )^{1/p} \nonumber \\ & \le C_{\beta,\beta'}\  (s-r)^{2\a-2\eps +1/q} \ \|\xi\|_{\mathbf{L}^{2p}(0{,}t)}^2
\end{align}
with
\begin{align*}
& C_{\beta,\beta'} = \left [ \frac{\beta(\beta'-2)}{\beta'-\beta}\right ]^{\frac{\beta'-2}\beta}\ .
\end{align*}
Similarly we obtain the following estimation for $I_2$:
\begin{align}\label{i12}
    I_2 & \le C_{\beta,\beta'}\  r^{2\a-2\eps +1/q} \ \|\xi\|_{\mathbf{L}^{2p}(0{,}t)}^2\ .
\end{align}
Reporting \eqref{i11} and \eqref{i12} in \eqref{i1} yields
\begin{align}\label{ii}
    \int_0^t | g_{s,r}(\tau)|^2 d\tau
 \le & \ 2\ C_{\beta,\beta'}\  (s-r)^{2\eps}\ t^{2\a-2\eps +1/q} \ \|\xi\|_{\mathbf{L}^{2p}(0{,}t)}^2\ .
\end{align}
On the event $\mathcal A_t = \{ (\int_0^t \xi_\tau^{2p} d\tau )^{1/p} \le \nu_t\}$ it holds that
\begin{align*}
    \int_0^t | g_{s,r}(\tau)|^2 d\tau
 \le & \ 2\ C_{\beta,\beta'}\  (s-r)^{2\eps}\ t^{2\a-2\eps +1/q} \ \nu_t\ .
\end{align*}
Now it is clear that the function $p$ must be defined as
\begin{align}\label{p}
p(y) & = (2C_{\beta,\beta'})^{1/2}t^{\a-\eps +1/2q} \nu_t^{1/2} y^{\eps}
\end{align}
and for fixed $r<s$, we consider the martingale $M=(M_u)_{0\le u\le t}$ defined by
\begin{align}\label{Mu}
& M_u = \int_0^u \frac{g_{s,r}(\tau)}{p(s-r)}\,dW_\tau \ .
\end{align}
Its quadratic variation satisfies for any $0\le u\le t$
$$ \langle M\rangle_u \le \int_0^t \frac{|g_{s,r}(s)|^2}{p(s-r)}\,ds \le 1 $$
almost-surely on $\mathcal A_t$. Let $W$ the Dambis, Dubins-Schwarz Brownian motion associated to the martingale $M$ such that $M_u=W_{\langle M\rangle_u}$. We have
\begin{align*}
   \mathbf E \left [ \Psi\left ( \frac{M_s^{(\alpha)}-M_r^{(\alpha)}}{p(|s-r|)} \right ) \1_{\mathcal A_t}  \right ]
   & = \mathbf E \left [ \exp\left ( \frac{M_t^2}4 \right ) \1_{\mathcal A_t}  \right ] \\
   & \le \mathbf E \left [ \exp\left ( \frac{1}4 \sup_{0\le r\le 1}|W_r|^2 \right ) \1_{\mathcal A_t}  \right ] \\
   & \le 2^{1/2} \ .
\end{align*}
Consequently $\mathbf{E} (B\1_{\mathcal A_t}) \le 2^{1/2} t^2$ and
\begin{align*}
& \mathbf{E}(\exp\{ \1_{\mathcal A_t} \ln^+(B)\}) \le 1+\mathbf E(\1_{\mathcal A_t}\exp\{  \ln^+(B)\}) \le 2+2^{1/2}t^2 :=C_t\ .
\end{align*}
We use the inequality
\begin{align}
\label{kappa}
\int_0^t \left (\ln^+(y^{-2})\right )^{1/2} y^{\eps-1}dy \le 2^{1/2}\int_0^{+\infty} z^{1/2} e^{-\eps z} dz = \frac{1}{\eps}\sqrt{\frac{\pi}{2\eps}}:=\kappa\ ,
\end{align}
in order to rewrite \eqref{garsia2} (with $p$ defined in \eqref{p}) as
\begin{align}\label{garsia3}
\sup_{0\le s\le t}| M_s^{(\alpha)}| & \le 2^{9/2} \left [ \left ( \ln^+(B)\right )^{1/2}t^{\eps}  +  \kappa \right ]
(2C_{\beta,\beta'})^{1/2}t^{\a-\eps +1/2q} \nu_t^{1/2} \nonumber \\
& \le \left [ \left ( \ln^+(B)\right )^{1/2}t^{\eps}  +  \kappa \right ] \times c_1(t)
\end{align}
with $c_1(t) = 32\,C_{\beta,\beta'}^{1/2}\,t^{\a-\eps +1/2q}\, \nu_t^{1/2}$.
No we end the proof with Chebishev's exponential inequality.
For $L\ge 1$ we have
\begin{align*}
    \p \left ( \sup_{0\le s\le t}| M_s^{(\alpha)}| \ge 2 L \kappa c_1(t) \ ,\ \mathcal A_t   \right )  & \le
    \p \left ( \left \{ \ln^+(B) \ge \frac{1}{t^{2\eps}} \left ( \frac{2 L\kappa c_1(t)}{c_1(t)}-\kappa\right )^2  \right \} \ \cap \mathcal A_t \right ) \\
    & \le  \mathbf E \big [ \exp\left ( \1_{\mathcal A_t} \ln^+(B) \right ) \big ] \ \exp\left \{  -\left ( \frac{\kappa\,(2L-1)}{t^{\eps}}\right )^2  \right \} \\
      &\le  C_t \ \exp\left \{  -\frac{\kappa^2\,L^2}{t^{2\eps}}  \right \}\ .
\end{align*}
We recall that $4\eps=2\alpha+1/q = 2(\beta'-\beta)/\beta\beta'$ and the expression \eqref{inegexp} is a consequence of the above inequality with the notation
\begin{align}
\label{c1}
{c}_1 & = 2^{11/2}\, \pi^{1/2}\, \left ( \frac{\beta\beta'}{\beta'-\beta}\right )^{3/2}\times \left [ \frac{\beta(\beta'-2)}{\beta'-\beta}\right ]^{\frac{\beta'-2}{2\beta}}\ .
\end{align}
\subsubsection*{Case (ii): we assume Hypothesis \ref{hyp-M} and  $\a>0$}
By Hypothesis \ref{hyp-M}, $\int_0^t \xi_\tau^2 d\tau $ exists almost-surely. Using \eqref{inc}, we replace \eqref{i0} by \begin{align}\label{iii}
    \int_0^t | g_{s,r}(\tau)|^2 d\tau
 \le & (s-r)^{2\a} \int_r^s \xi_\tau^2 d\tau + (s-r)^{2\eps}\int_0^r (r-\tau)^{2\a-2\eps}\xi_\tau^2 d\tau   \nonumber \\ %
 \le & (s-r)^{2\a} \int_r^s \xi_\tau^2 d\tau + (s-r)^{2\eps}\ r^{2\a-2\eps}\int_0^r \xi_\tau^2 d\tau   \nonumber \\ %
\le & 2(s-r)^{2\eps}\ t^{2\a-2\eps}\int_0^t \xi_\tau^2 d\tau
\end{align}
with $0<\eps<\a$. The rest of the proof is similar with the following modifications. We use the martingale $M$ defined by \eqref{Mu} with the new function $p$ defined by $p(y)= 2^{1/2}t^{\alpha-\eps}\nu_t^{1/2}y^\eps$. On the event $\{ \int_0^t \xi_\tau^{2} d\tau  \le \nu_t\}$, $M$ has also a quadratic variation bounded by $1$. The inequality \eqref{garsia3} is replaced by
\begin{align*}
\sup_{0\le s\le t}| M_s^{(\alpha)}|
& \le \left [ \left ( \ln^+(B)\right )^{1/2}t^{\eps}  +  \kappa \right ] \times 32\,t^{\a-\eps}\, \nu_t^{1/2}
\end{align*}
where $\kappa$ is defined in \eqref{kappa}. The rest of the proof is identical.
\subsubsection*{Case (iii): we assume that $\xi$ is bounded}
When $\a\in (-\frac{1}2{,}0)$, there exists $\eps>0$ such that $1+2\a-2\eps >0$. Then from \eqref{i0}, it is easy to see that \eqref{ii} may be replaced by
\begin{align}\label{iibis}
    \int_0^t | g_{s,r}(\tau)|^2 d\tau
 \le & \ 2\, c_{\infty}^2\,  (t-r)^{2\eps}\, t^{1+2\a-2\eps}.
\end{align}
When $\a>0$, \eqref{iii} may also be replaced by \eqref{iibis}.

The rest of the proof is similar to the previous case with
the help of the function $p$ defined by $p(y)= (2c_\infty)^{1/2}t^{1/2+\alpha-\eps}y^\eps$.
\end{proof}

\section{Proof of Proposition \ref{conv00}}\label{proof_conv00}
The result is based of the following fractional version of the Toeplitz lemma.
\begin{lemme}\label{toeplitz}
Let $\a>0$. Let $(x_t)_{t\ge 0}$ be a continuous real function such that $\lim_{t\to\infty}x_t = x$ and let  $(\gamma_t)_{t\ge 0}$ be a measurable and positive. Then it holds that
$$\frac{\int_0^t (t-s)^{\alpha-1} \left ( \int_0^s\gamma_r dr \right ) x_s ds  }{\int_0^t (t-s)^{\alpha-1} \left ( \int_0^s\gamma_r dr \right ) ds } \xrightarrow[t\to\infty]{}x, $$
provided that $\lim_{t\to\infty}  \int_0^t \gamma_sds =+\infty$.
\end{lemme}
\begin{proof}
Let $\eps>0$ and $A$ be such that $|x_s-x|<\eps$ for $s>A$. We denote $C_A=\sup_{s\le A}|x_s-x|$.
By Fubini's theorem
$$  \int_0^t (t-s)^\alpha \gamma_s  ds  =\alpha\int_0^t (t-s)^{\alpha-1}\left ( \int_0^s \gamma_r dr\right ) ds \ , $$
and we write for $t>A$
\begin{align}\label{ioio}
\left |\frac{\int_0^t (t-s)^{\alpha-1} \left ( \int_0^s\gamma_r dr \right ) x_s ds  }{\int_0^t (t-s)^{\alpha-1} \left ( \int_0^s\gamma_r dr \right ) ds } - x \right | & \le \frac{\int_0^t (t-s)^{\alpha-1} \left ( \int_0^s\gamma_r dr \right ) | x_s-x | ds  }{\int_0^t (t-s)^{\alpha-1} \left ( \int_0^s\gamma_r dr \right ) ds } \nonumber \\ %
& \le \eps + C_A\ \frac{\int_0^A (t-s)^{\alpha-1} \left ( \int_0^s\gamma_r dr \right ) ds  }{\int_0^t (t-s)^{\alpha-1} \left ( \int_0^s\gamma_r dr \right ) ds } \ .
\end{align}
Another application of Fubini's theorem implies that
\begin{align*}
  \frac{\int_0^A  \left ( \int_r^A(t-s)^{\alpha-1} ds \right ) \gamma_rdr  }{\int_0^t  \left ( \int_r^t(t-s)^{\alpha-1} ds \right ) \gamma_rdr } & =  \frac{\int_0^A \gamma_r \left [ (t-r)^\a -(t-A)^\a \right ] dr }{\int_0^t (t-r)^\a \gamma_rdr} \nonumber \\
& \le  \frac{\int_0^A (t-r)^\a \gamma_r dr }{\int_0^t (t-r)^\a \gamma_rdr}  \\ &
\le  \frac{t^\a \int_0^A \gamma_r dr }{\int_0^{t/2} (t-r)^\a \gamma_r dr } \nonumber \\
& \le  \frac{t^\a \int_0^A \gamma_r dr }{(t/2)^\a \int_0^{t/2}  \gamma_r dr }
\end{align*}
and the last term tends to $0$ as $t$ tends to $\infty$. We report this convergence in \eqref{ioio} and we obtain the result.
\end{proof}
Now we prove \eqref{conv0}.
\begin{proof}
By the stochastic Fubini theorem
$$ \int_0^t(t-s)^\a \xi_sdW_s =\a\int_0^t (t-s)^{\a-1} \left ( \int_0^s \xi_rdW_r \right ) ds $$ and consequently
\begin{align*}
 \frac{\int_0^t (t-s)^\alpha \xi_s dW_s}{\int_0^t (t-s)^{2\alpha} \xi_s^2 ds} & =  \frac{\int_0^t (t-s)^{\a-1}\left ( \int_0^s \xi_r^2dr\right )  \frac{\int_0^s \xi_rdW_r}{\int_0^s \xi_r^2dr}   ds }{\int_0^t (t-s)^{\a-1}\left ( \int_0^s \xi_r^2dr\right )  ds}.
\end{align*}
Since it is assumed that $\int_0^{\infty} \xi_s^2ds = +\infty$ almost-surely,
$$ \frac{\int_0^s \xi_r dW_r}{\int_0^s \xi_r^2 dr} \xrightarrow[t\to\infty]{a.s.}0 $$
and the generalized Toeplitz lemma \ref{toeplitz} implies \eqref{conv0}.
\end{proof}

%
%

\end{document}